\documentclass[10pt]{article}
\usepackage{amssymb}
\usepackage{graphicx}
\usepackage{xcolor} 
\usepackage{tensor}
\usepackage{fullpage} 
\usepackage{amsmath}
\usepackage{amsthm}
\usepackage{verbatim}
\usepackage{enumitem}
\setlist[enumerate]{leftmargin=1.5em}
\setlist[itemize]{leftmargin=1.5em}

\definecolor{green}{rgb}{0,0.8,0} 

\newtheorem{theorem}{Theorem}[section]

\newtheorem{lemma}[theorem]{Lemma}
\newtheorem{proposition}[theorem]{Proposition}

\theoremstyle{definition}

\theoremstyle{remark}
\newtheorem{remark}[theorem]{Remark}

\numberwithin{equation}{section}

\providecommand{\bysame}{\leavevmode\hbox to3em{\hrulefill}\thinspace}
\providecommand{\MR}{\relax\ifhmode\unskip\space\fi MR }

\providecommand{\href}[2]{#2}
\newcommand{\nrm}[1]{\Vert#1\Vert}

\newcommand{\nnrm}[1]{{\vert\kern-0.25ex\vert\kern-0.25ex\vert #1 
    \vert\kern-0.25ex\vert\kern-0.25ex\vert}}

\newcommand{\lap}{\Delta}

\newcommand{\rd}{\partial}
\newcommand{\nb}{\nabla}

\newcommand{\alp}{\alpha}
\newcommand{\bt}{\beta}
\newcommand{\gmm}{\gamma}
\newcommand{\Gmm}{\Gamma}
\newcommand{\dlt}{\delta}

\newcommand{\eps}{\epsilon}

\newcommand{\tht}{\theta}

\newcommand{\omg}{\omega}
\newcommand{\Omg}{\Omega}

\newcommand{\bbA}{\mathbb A}

\newcommand{\bbR}{\mathbb R}

\newcommand{\bbT}{\mathbb T}

\newcommand{\bbZ}{\mathbb Z}

\newcommand{\pr}{\partial}

\begin{document}

\title{Loss of regularity for the 2D Euler equations}
\author{In-Jee Jeong\thanks{Department of Mathematical Sciences and RIM of Seoul National University. E-mail: injee\_j@snu.ac.kr} }

\date{\today}




\maketitle


\begin{abstract}
	In this note, we construct solutions to the 2D Euler equations which belong to the Yudovich class but lose $W^{1,p}$ regularity continuously with time. 
\end{abstract}


\section{Introduction}

The dynamics of inviscid and incompressible fluid is described by the Euler equations: given some $n$-dimensional domain $\Omg$, the system is given by \begin{equation}  \label{eq:Euler}
\left\{
\begin{aligned}
\rd_t u + u\cdot\nb u +\nb p & = 0, \\
\nb\cdot u &= 0, 
\end{aligned}
\right.
\end{equation} where $u(t,\cdot):\Omg\rightarrow\bbR^n$ and $p(t,\cdot):\Omg\rightarrow\bbR$ denote the velocity and pressure of fluid at time $t$, respectively. When $\Omg$ has boundary, \eqref{eq:Euler} should be supplemented with the slip boundary condition $u(t,x)\cdot n(x) = 0$ for $x\in \partial\Omg$, where $n(x)$ is the unit normal vector. 

We shall be concerned with the two-dimensional Euler equations on $\bbT^2 = (\bbR/\bbZ)^2$; introducing the vorticity $\omg = \nb\times u$ and taking the curl of \eqref{eq:Euler}, we have the vorticity form of the 2D Euler equations: \begin{equation}\label{eq:Euler-vort}
\begin{split}
\rd_t \omg +u\cdot\nb\omg = 0, \quad u = \nb^\perp\lap^{-1}\omg. 
\end{split}
\end{equation} Here, $\nb^\perp = (-\rd_{x_2},\rd_{x_1})^\top$. The goal of this note is to construct  solutions to \eqref{eq:Euler-vort} which belongs to a well-posedness class yet lose Sobolev regularity with time. 

To motivate our results, we briefly review the well-posedness theory for the Euler equations \eqref{eq:Euler}. For sufficiently nice $n$-dimensional domains, \eqref{eq:Euler} is locally well-posed in $C^{k,\alp}\cap L^2$ with $k\ge 1$ and $0<\alp<1$ and $W^{s,p}$ with $s>\frac{n}{p}+1$. That is, for $u_0$ belonging to such a space, there exist $T>0$ and a unique local-in time solution \eqref{eq:Euler} such that $u(t)$ belongs to the same space for all $0\le t<T$ and $u(0)=u_0$. It is known that these regularity requirements are sharp; for $u_0$ less regular, the regularity may not propagate in time for a solution of \eqref{eq:Euler}. To be more precise, for $n= 3$, $s=1$, and any $1\le p$, there exist examples of solutions $u(t)$ such that \begin{equation}\label{eq:loss-vel}
\begin{split}
u(0) \in W^{s,p}(\bbT^n) \quad \mbox{and} \quad  u(t) \notin W^{s,p}(\bbT^n) \quad \mbox{for any}\quad t>0. 
\end{split}
\end{equation} Similar examples exist which do not propagate initial $C^\alp$ regularity of $u(0)$ for any $\alp<1$. We shall recall the constructions below; for now, let us just note that the examples are based the so-called $2+\frac{1}{2}$ flow, which gives the restriction that $n\ge 3$ (Bardos-Titi \cite{BT,BT2}).   It seems that in the two dimensional case, solutions satisfying \eqref{eq:loss-vel} have not been constructed before.

Recently, Bourgain-Li \cite{BL1,BL2,BL3D} established ill-posedness of \eqref{eq:Euler} at critical regularity; roughly speaking, they were able to prove that \eqref{eq:loss-vel} occurs with $s=\frac{n}{p}+1$ and $p>1$, in dimensions two and three. (See also \cite{EJ,EM1,JY2,MY,JY,JKi} for simpler proofs and further developments.) It does not imply that \eqref{eq:loss-vel} occurs also in less regular Sobolev spaces; in principle, this question could be harder since one expects less control over weaker solutions. 


Our main result shows that an $L^\infty \cap W^{1,p}$-vorticity with $p<2$ may continuously lose integrability with time. For bounded initial vorticity, the uniqueness and existence of the solution $\omg(t,x)\in L^\infty([0,\infty);L^\infty(\bbT^2))$ is provided by the celebrated Yudovich theory \cite{Y1}. Moreover, we shall take the initial data $\omg_0$ to be Lipschitz continuous away from a single point, which is a property preserved by Yudovich solutions (see \cite{EJ} for a proof). In particular, for any $t>0$, $\nb\omg$ is well-defined almost everywhere in space and bounded away from a single point. 

\begin{theorem}\label{thm:torus}
	There exist  $1\le p^*<2$ and $c_0>0$ such that for any $p^*<p_0<2$, we can find an initial data $\omega_0 \in L^\infty \cap  \left( \cap_{p<p_0}  W^{1,p}\right) (\mathbb{T}^2)$ such that the unique solution $\omg(t)$ satisfies, with some  {$T^*=T^*(p_0)>0$}, \begin{equation*}
	\begin{split}
	\nrm{\omega(t,\cdot)}_{W^{1, {q(t)}}(\bbT^2)} = + \infty~,
	\end{split}
	\end{equation*} for  \begin{equation*}
	\begin{split}
	 {q(t):=} 1 + \frac{1}{\frac{1}{p_0-1} + c_0t},\quad 0\le t\le T^*.
	\end{split}
	\end{equation*} 
\end{theorem}

\begin{remark}
	Let us give a few  remarks on the statement. 
	\begin{itemize}
		\item For any $t>0$, we may define the index $p(t)=\sup\{0< q: \omg(t)\in W^{1,q}(\bbT^2)  \} $. While the statement of Theorem \ref{thm:torus} does not exclude the possibility that $p(t)$ jumps downwards in time, this behavior is impossible due to the bound given in Proposition \ref{prop:E}. Hence, for the data $\omg_0$ given in Theorem \ref{thm:torus}, the index $p(t)$ is continuous in time and satisfies $p(t)<p_0$ for any $t\in (0,T^*]$. 
		
		\item In terms of the velocity, Theorem \ref{thm:torus} says that \eqref{eq:loss-vel} occurs  {at least for a small interval of time with} $u_0\in W^{2,p}(\bbT^2)$ with $p^*<p<2$. 
		\item The restriction $p^*<p_0$ should be technical but it could take significant work to remove it; we shall see in the proof that $p^*$ depends only on the constant $C$ in Lemma \ref{lem:key}.
		\item Our data can be localized to any small ball and hence the specific choice of domain $\bbT^2$ is not important. 
	\end{itemize}
\end{remark}


\subsection*{Organization of the paper}

The rest of this paper is organized as follows. In Section \ref{sec:example}, we present several examples and discussions which illustrate delicate nature of Euler equations at low regularity. The proofs are then presented in Section \ref{sec:proofs}.

\section{Examples}\label{sec:example}

We present several examples in order to give a sense of the behavior described in the main results. In \ref{subsec:BT}, we recall the well-known construction of Bardos-Titi in three dimensions and make a comparison with our result. Then in \ref{subsec:BC}, we present a time-independent vector field which is able to cause the phenomenon of continuous loss of Sobolev regularity for the advected scalar. In the same section we comment on some difficulties in actually using the vector field in the context of the Euler equations. 

\subsection{Loss of regularity in shear flows}\label{subsec:BT}

Any solution of the 2D Euler equations can be lifted to solutions in 3D (and higher); given $(u^{2D},p^{2D})$ solving \eqref{eq:Euler} in $\bbT^2$, define \begin{equation*}
\begin{split}
u(t,x) = (u^{2D}(t,x_1,x_2),u_3(t,x_1,x_2))
\end{split}
\end{equation*} where $u_3$ is any solution to \begin{equation*}
\begin{split}
\rd_t u_{3} + u^{2D}\cdot\nb u_3 = 0. 
\end{split}
\end{equation*} Then one can see that $u$ defines a solution to the 3D Euler equations with $p = p^{2D}$. This is sometime referred to as the $2+\frac{1}{2}$-dimensional flow (see \cite{MB}). A special class of $2+\frac{1}{2}$-dimensional flows are given by the following \textit{shear flows}: \begin{equation}\label{eq:shear}
\begin{split}
u(t,x) = (u_1(x_2),0,u_3(x_1-tu_1(x_2),x_2)). 
\end{split}
\end{equation} Note that this defines a solution to the 3D Euler equations in $\bbT^3$ with zero pressure for any reasonably smooth functions $u_1$ and $u_3$ of one variable.\footnote{For $u_1$ and $u_3$ belonging to $L^2(\bbT)$, it can be shown that $u$ defined in \eqref{eq:shear} is a weak solution to the 3D Euler equations; see \cite[Theorem 1.2]{BT}.} The form \eqref{eq:shear} was introduced in a work of DiPerna-Majda \cite{DiPM} to provide an example of weak solution sequence in 3D Euler whose limit is \textit{not} a solution to 3D Euler. For more applications and references regarding this flow, one can see the illuminating papers of Bardos-Titi \cite{BT,BT2}.

\begin{proposition}[{{see \cite[Proposition 3.1]{BT2} and \cite[Theorem 2.2]{BT}}}]
	There exists initial data $u_0 \in W^{1,p}(\bbT^3)$ with any $1\le p$ ($u_0\in C^\alp(\bbT^3)$ with any $0<\alp<1$, resp.) of the form \begin{equation*}
	\begin{split}
	u_0(x) = (u_1(x_2),0,u_3(x_1,x_2))
	\end{split}
	\end{equation*} such that the corresponding shear flow solution $u(t)$ given in \eqref{eq:shear} does not belong to $W^{1,p}(\bbT^3)$ ($C^\alp(\bbT^3)$, resp.) for any $t>0$.
\end{proposition}

\begin{proof}
	We only consider the $W^{1,p}$-case with $1\le p<\infty$. Note that \begin{equation*}
	\begin{split}
	\nrm{u_0}_{W^{1,p}}^p \lesssim \int_0^1 |\rd_{x_2}u_1(x_2)|^p dx_1 + \int_0^1\int_0^1 |\rd_{x_1}u_3(x_1,x_2)|^p + |\rd_{x_2}u_3(x_1,x_2)|^p dx_1dx_2
	\end{split}
	\end{equation*} and  {computing $\rd_{x_2}u(t)$ using \eqref{eq:shear}, we obtain that \begin{equation}\label{eq:comp}
	\begin{split}
	\nrm{u(t)}_{W^{1,p}}^p\gtrsim -\nrm{\rd_{x_2}u_3}_{L^p}^{p} +  t^p\int_0^1\int_0^1 |\rd_{x_2}u_1(x_2)|^p|\rd_{x_1}u_3(x_1,x_2)|^pdx_1dx_2. 
	\end{split}
	\end{equation} Given the above formula, let us define $u_1(x_2) =  |x_2|^{1-\frac{1}{p}+\eps}$  for small $\eps>0$ near $x_2=0$ and smooth otherwise. Next, define $u_3(x_1,x_2)=|x|^{1-\frac{2}{p}+\eps}$ near $|x|=0$, where $|x|=\sqrt{x_1^2+x_2^2}$. Since $|\rd_{x_i}u_3| \lesssim |x|^{-\frac{2}{p}+\eps}$ for $i=1,2$, it is clear that $u_0 \in W^{1,p}$. In particular, the first term on the right hand side of \eqref{eq:comp} is bounded.} Furthermore, employing polar coordinates $(r,\tht)$,  \begin{equation*}
	\begin{split}
	\int_0^1\int_0^1 |\rd_{x_2}u_1(x_2)|^p|\rd_{x_1}u_3(x_1,x_2)|^pdx_1dx_2 & \gtrsim \int_0^{r_0} r^{-1+\eps p} r^{-2+\eps p} r dr =+\infty
	\end{split}
	\end{equation*} for $\eps <p^{-1}$. This finishes the proof. 
\end{proof}

Note that for solutions of the form \eqref{eq:shear}, the integrability index may drop instantaneously at $t =0$ but $u(t_1)$ and $u(t_2)$ have the same Sobolev regularity for any $t_1,t_2>0$. (One can see that $u(t)$ belongs to at least $W^{1,\frac{p}{2}}(\bbT^3)$ for any $t>0$.) Similar phenomenon occurs in the scale of $C^\alp$ spaces, see \cite{BT}. This behavior is very different from continuous loss of integrability that our solutions exhibit. Actually, in two dimensions, conservation of $\omg\in L^\infty$ \textit{prohibits} jump of the integrability index, as the following remarkable result due to Elgindi shows: \begin{proposition}[{{see \cite[Lemma 3]{EJ}}}]\label{prop:E}
	Let $\omg_0\in (L^\infty\cap W^{1,p})(\bbT^2)$ with $0<p\le 2$. Then, the Yudovich solution corresponding to $\omg_0$ satisfies \begin{equation*}
	\begin{split}
	\nrm{\omg(t)}_{W^{1,q(t)}} \le \nrm{\omg_0}_{W^{1,p}}
	\end{split}
	\end{equation*} where $q(t)$ is the solution to the ODE \begin{equation*}
	\begin{split}
	\dot{q}(t) = -C\nrm{\omg_0}_{L^\infty} q(t)^2 ,\quad q(0) = p
	\end{split}
	\end{equation*} with some universal constant $C>0$ not depending on $p$. 
\end{proposition}

\subsection{Loss of regularity with the Bahouri-Chemin background}\label{subsec:BC}

We shall consider the following vector field $v$ on the positive quadrant $(\bbR_+)^2$: \begin{equation}\label{eq:BC-model}
\begin{split}
v_1(x_1,x_2) = -x_1 \ln \frac{1}{x_2}~, \qquad v_2(x_1,x_2) = x_2 \ln \frac{1}{x_2}~.
\end{split}
\end{equation} This is a toy model for the Bahouri-Chemin velocity field which will be introduced below. Although $v$ is not exactly divergence free (the divergence satisfies $\mathrm{div}(v)=-1$), one can add $x_2$ to $v_2(x_1,x_2)$ to make it divergence-free. (The resulting flow can be shown to demonstrate the same behavior but computations are more tedious.) We have the following 
\begin{proposition}\label{prop:BC-model-flow}
	Let $f_0$ be the function defined in $(\bbR_+)^2$ by $f_0(r,\tht)=\sin(2\tht)$ for $0\le r<1$ in polar coordinates. Then, the solution $f(t)$ to the transport equation  { \begin{equation}\label{eq:transport-BC-model}
	\left\{
	\begin{aligned}
		\rd_t f + v \cdot\nb f = 0, 	& \\
		f(t=0)=f_0 &
	\end{aligned}
	\right.
\end{equation}}satisfies \begin{equation}\label{eq:loss-model}
	\begin{split}
	\nrm{f(t)}_{W^{1,q(t)}} = +\infty
	\end{split}
	\end{equation} and \begin{equation}\label{eq:retain-model}
	\begin{split}
	\nrm{f(t)}_{W^{1,q(t)-\eps}}<\infty, \mbox{ for any } \eps>0
	\end{split}
	\end{equation} for all $t>0$, where $q(t) := \frac{2}{2-\exp(-t)}$. 
\end{proposition}

Before we start the proof, let us observe a convenient lemma which gives a lower bound on the $W^{1,p}$ norm based on the distance between two level sets. 

\begin{lemma}\label{lem:lvl-Sob}
	Assume that $f\in L^\infty(\bbR^2)\cap Lip(\bbR^2\backslash\{0\})$ written in polar coordinates satisfies $f(r,0) = 0$ and $f(r,\tht^*(r))=1$ for all $0<r<r_0$ with some $r_0>0$ where $\tht^*:[0,r_0]\rightarrow \bbR/2\pi\bbZ$ is a measurable function. Then, \begin{equation*}
	\begin{split}
	\nrm{\nb f}_{L^p}^p \ge c \int_0^{r_0} r^{1-p} (\tht^*(r))^{1-p} dr.
	\end{split}
	\end{equation*}
\end{lemma}
\begin{proof}
	The $L^p$-norm of $\nb f$ is equivalent, up to absolute constants, with \begin{equation}\label{eq:Lp-def}
	\begin{split}
	 \nrm{\rd_r f}_{L^p}^p + \nrm{r^{-1}\rd_\tht f}_{L^p}^p. 
	\end{split}
	\end{equation} Note that for any fixed $0<r<r_0$, \begin{equation*}
	\begin{split}
	1 = f(r,\tht^*(r))-f(r,0) = \int_0^{\tht^*(r)} \rd_\tht f (r,\tht') d\tht' \le |\tht^*(r)|^{1-\frac{1}{p}} \left( \int_0^{2\pi} |\rd_\tht f(r,\tht)|^p d\tht \right)^{\frac{1}{p}},
	\end{split}
	\end{equation*} which gives \begin{equation*}
	\begin{split}
	 \int_0^{2\pi} |\rd_\tht f(r,\tht)|^p d\tht  \ge |\tht^*(r)|^{1-p}. 
	\end{split} 
	\end{equation*} Hence \begin{equation*}
	\begin{split}
	\nrm{\nb f}_{L^p}^p \ge c \int_0^{r_0} \int_0^{2\pi} r^{1-p}     |\rd_\tht f(r,\tht)|^p d\tht  dr \ge c \int_0^{r_0} r^{1-p} (\tht^*(r))^{1-p} dr. \qedhere
	\end{split}
	\end{equation*} 
\end{proof}

\begin{proof}[Proof of Proposition \ref{prop:BC-model-flow}]
	We fix some $0<a<1$ and let $\phi(t)=(\phi_1(t),\phi_2(t))$ be the trajectory of the point $(a,a)$ by the flow generated by $v$. Then, from \begin{equation*}
	\begin{split}
	\dot{\phi}_2 = \phi_2 \ln\frac{1}{\phi_2}, 
	\end{split}
	\end{equation*} we have \begin{equation*}
	\begin{split}
	\ln\frac{1}{\phi_2} = e^{-t} \ln\frac{1}{a}.
	\end{split}
	\end{equation*} Next, from \begin{equation*}
	\begin{split}
	\dot{\phi}_1 = -\phi_1\ln\frac{1}{\phi_2} = -\phi_1 e^{-t} \ln\frac{1}{a},
	\end{split}
	\end{equation*} we obtain that \begin{equation*}
	\begin{split}
	\ln\frac{1}{\phi_1} = (2-e^{-t}) \ln\frac{1}{a}. 
	\end{split}
	\end{equation*}  {This shows that the image by the flow map $\Phi(t,\cdot)$ of the diagonal segment $\{ (a,a):0<a<1 \}$ can be parameterized by $$\Gmm(t):=\{ (x_1,x_1^{\frac{\exp(-t)}{2-\exp(-t)}}) : 0<x_1<1 \}$$ for any $t>0$.} (See Figure 1.) That is, the solution $f$ with initial data $f_0(r,\tht)=\sin(2\tht)$ for $r\le1$ satisfies $f(t,\Gmm(t))\equiv 1$. Since $f(t,(0,x_2))\equiv 0$ for all $t$, \eqref{eq:loss-model} follows from a computation similar to the one given in the proof of Lemma \ref{lem:lvl-Sob}: for any fixed $t>0$ and $0<x_2<1$, \begin{equation*}
	\begin{split}
	\int_0^1 |\rd_{x_1}f (x_1,x_2)|^p dx_1 \ge x_2^{-\frac{1}{\gmm}(p-1)}, \quad \gmm = \frac{\exp(-t)}{2-\exp(-t)},
	\end{split}
	\end{equation*} and hence \begin{equation*}
	\begin{split}
	\int_0^1\int_0^1 |\rd_{x_1}f (x_1,x_2)|^p dx_1 dx_2\ge \int_0^1 x_2^{-\frac{1}{\gmm}(p-1)} dx_2,
	\end{split}
	\end{equation*} which is integrable if and only if $-\frac{1}{\gmm}(p-1)>-1$. We omit the proof of \eqref{eq:retain-model}. 
\end{proof}

	\begin{figure}
		\centering
	\includegraphics[scale=0.5]{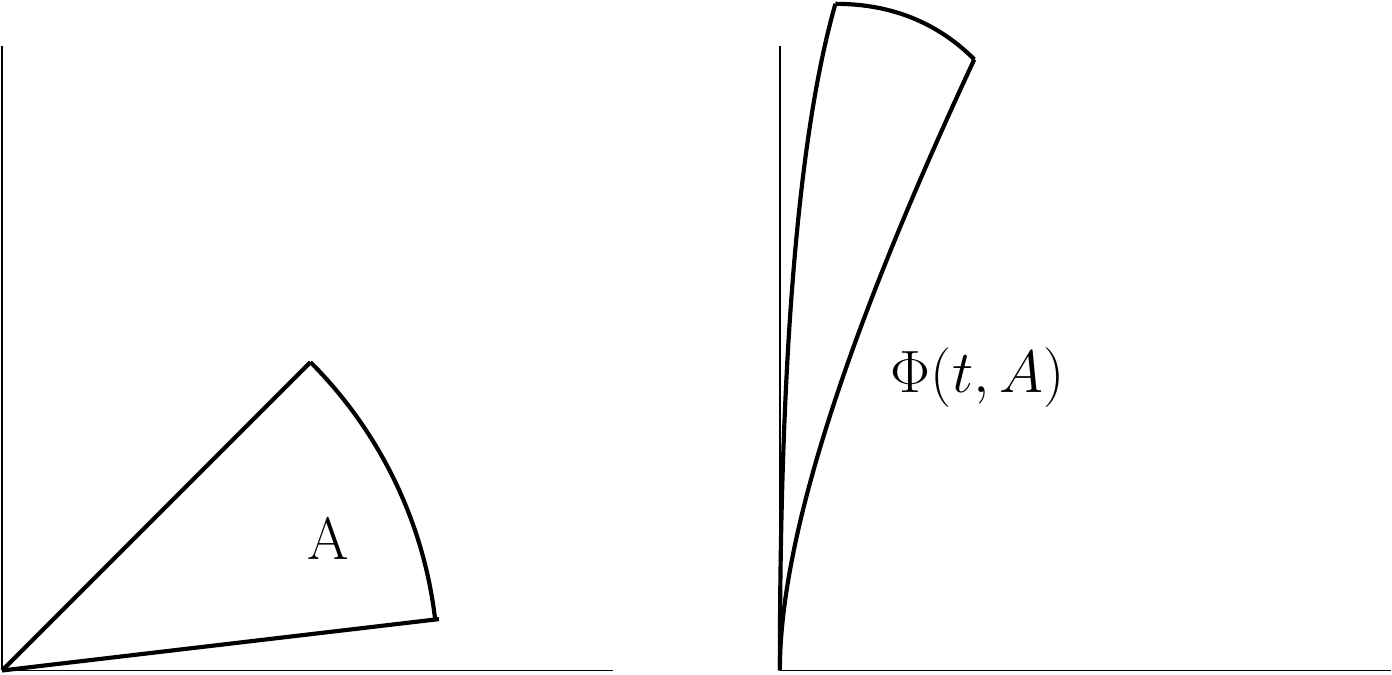}
	\caption{Flow of a ``rectangle'' in polar coordinates: the curves $\{ \tht = \mathrm{const} \}$ becomes instantaneously tangent with the $x_2$-axis for $t>0$.}
\end{figure}

In the explicitly solvable example above, continuous loss of Sobolev regularity occurs from the cusping of a level set for the advected scalar. While it is correct that the velocity corresponding to $\omg_0 = \sin(2\tht)$ has approximately the form \eqref{eq:BC-model} (cf. Lemma \ref{lem:key}), the velocity field immediately changes for $t>0$ in the 2D Euler case, which is a nonlinear problem. Indeed, cusping of the level sets weakens the velocity gradient as some chunk of vorticity moves away from the origin (the integral $I(x)$ in \eqref{eq:key} is reduced), which at the same time slows down the cusping phenomenon. To illustrate this point, one can simply compute the 2D Euler velocity gradient corresponding to $f(t)$ where $f$ is the solution to \eqref{eq:transport-BC-model}. At the origin (which is the point $\nb \nb^\perp \lap^{-1}f$ is supposedly most singular),  \begin{equation*}
\begin{split}
\rd_{x_1}\rd_{x_2}\lap^{-1}f(t,0) \simeq c\int \frac{y_1y_2}{|y|^4} f_0\circ \phi(t,y) dy \simeq c'  \int \frac{y_1y_2}{|y|^4} \frac{\phi_1(y)}{\phi_2(y)} dy \simeq \frac{c''}{t}
\end{split}
\end{equation*} for $0<t\ll 1$, using that $\phi_1(t,y)\simeq y_1y_2^t$ and $\phi_2(t,y)\simeq y_2^{1-t}$. Hence, the velocity corresponding to $f$ becomes Lipschitz continuous instantaneously for $t>0$. 

It is an interesting problem by itself to determine the behavior of the Yudovich solution with initial data $\omg_0\sim \sin(2\tht)$ near the origin. In \cite[Section 6.2]{EJ2}, a formal nonlinear system which models this behavior was introduced. Roughly speaking, the model equation is obtained by replacing the 2D Euler velocity gradient with the main term $I(x)$ in \eqref{eq:key}. The formal model is still not explicitly solvable, but it can be reduced to a second order ODE system with time-dependent coefficients. The numerical solution suggests that cusping of level sets occur but the cusps are not algebraic (that is, of the form $(x_1,x_1^\gmm)$) but just logarithmic of the form $(x_1,x_1(\ln\frac{1}{x_1})^\gmm)$ for some time-dependent $\gmm$. The corresponding velocity indeed \textit{regularizes} with time; it can be argued that $\nrm{\nb u(t)}_{L^\infty}\simeq ct^{-1}$ for $0<t\ll 1$. This does not allow for continuous in time loss of Sobolev regularity for $\omg$. 

Closing this section, let us remark that the situation is simpler when the spatial domain has a boundary. For instance, one can consider instead of $\bbT^2$, the domain $\mathbb{T} \times [0,1]$ which has the exact same Biot-Savart law.\footnote{Note that any smooth solution on $\bbT^2$ can be regarded as a smooth solution on $\mathbb{T} \times [-1,1]$, but the converse holds only when the solution vanishes at the boundary $\bbT\times (\{0\}\cup\{1\})$.} We now have that $f_0 = \cos(\tht) \in W^{1,p}(\mathbb{T} \times [-1,1])$ for any $p<2$ and that under the flow given in \eqref{eq:BC-model}, $f(t)$ converges pointwise to $\mathrm{sgn}(x_1)$. In the context of the 2D Euler equation, this guarantees that the velocity corresponding to $\omg(t)$ with $\omg_0 = f_0$ retains the asymptotic form of \eqref{eq:BC-model} for all $t>0$. For this reason, a very short proof of Theorem \ref{thm:torus} in the $\mathbb{T} \times [-1,1]$-case can be obtained (see the author's thesis \cite[Corollary 2.2.7]{Jthesis}). Alternatively, one can adopt the argument of Zlatos \cite{Zm} who showed merging of level sets for 2D Euler solutions in the disc. 

\section{Proof}\label{sec:proofs}

\subsection{Preliminaries}

We recall the basic log-Lipschitz estimate for the velocity. For a proof, see \cite{MB,MP}. \begin{lemma}\label{lem:log-Lip}
	Let $u = \nb^\perp\lap^{-1}\omg$ with $\omg\in L^\infty(\bbT^2)$. Then, for any $x,x' \in \bbT^2$ with $|x-x'| < 1/2$, \begin{equation}\label{eq:log-Lip}
	\begin{split}
	|u(x)- u(x')| \le C\nrm{\omega}_{L^\infty} |x-x'| \ln \left(\frac{1}{|x-x'|}\right)~.
	\end{split}
	\end{equation} 
\end{lemma}
The key technical tool is the following lemma (see Kiselev--Sverak \cite{KS}, Zlatos \cite{Z}, Elgindi--Jeong \cite{EJ2,EJSVP2}), which is in some sense complementary to the previous log-Lipschitz estimate. \begin{lemma}\label{lem:key}
	Let $\omega \in L^\infty( \mathbb{T}^2 )$ be odd with respect to both axes. For any $x = (x_1,x_2)$ with $x_1,x_2 \in (0,1/2]$, we have \begin{equation}\label{eq:key}
	\begin{split}
	(-1)^j\frac{ u_j(x)}{x_j} = I(x)+ B_j(x) 
	\end{split}
	\end{equation} where \begin{equation}\label{eq:I-def}
	\begin{split}
	I(x):= \frac{4}{\pi} \int_{[2x_1,1]\times[2x_2,1]} \frac{y_1y_2}{|y|^4} \omega(y) dy
	\end{split}
	\end{equation}  and \begin{equation}\label{eq:B-est}
	\begin{split}
	\left| B_j(x)  \right| \le C \nrm{\omega}_{L^\infty}  \ln\left( 10 + \frac{x_{3-j}}{x_j} \right)  , \qquad j = 1, 2. 
	\end{split}
	\end{equation} 
\end{lemma}

\subsection{Proof of Theorem \ref{thm:torus}}

In the proof, we shall take the initial vorticity $\omg_0\in L^\infty$ to be non-negative on $[0,1]^2$ and has odd symmetry with respect to both axes; that is, \begin{equation*}
 \begin{split}
 \omega_0(x_1,x_2) = -\omega_0(-x_1,x_2) = -\omega_0(x_1,-x_2), \qquad x_1, x_2 \in \bbT = [-1,1).
 \end{split}
 \end{equation*} The unique solution retains both the sign property and odd symmetry. For this reason, it suffices to specify the solution (and the data) only in the region $[0,1]^2$. Moreover, we shall normalize $\nrm{\omg_0}_{L^\infty}=1$. 
 
Our initial data on $[0,1]^2$ will be even symmetric with respect to $x_1=x_2$, and in the region $0<x_1<x_2$, we define using polar coordinates  {with a small $\beta>0$ to be specified below:} \begin{equation*}
\begin{split}
\omg_0(r,\tht) = \chi(r) \begin{cases}
r^{-\beta}\tht &\quad 0\le\tht <r^\beta,\\
1 &\quad r^\beta\le \tht. 
\end{cases} 
\end{split}
\end{equation*} Here, $\chi(r)$ is a smooth cutoff function satisfying $\chi(r)=1$ for $0<r<\frac{1}{2}$. Using \eqref{eq:Lp-def}, it is easy to see that $\omg_0 \in W^{1,p}$ if and only if $p<1+\frac{1}{1+\beta}$. Therefore, we  {may define} $\beta$ to satisfy $p_0 = 1+\frac{1}{1+\beta}$, where $p_0<2$ is given in the statement of Theorem \ref{thm:torus}. 
 
To begin with,  {we introduce the flow map: given $x$, $\Phi(t,x)$ is defined by the solution to the ODE \begin{equation*}
		\begin{split}
			\frac{d}{dt}\Phi(t,x) = u(t,\Phi(t,x)), \qquad \Phi(0,x)=x.
		\end{split}
	\end{equation*}
}
Now, integrating \eqref{eq:log-Lip} in time, one can obtain the following flow estimate which is valid for $0<t<t_0$ and $|x-x'|<\frac{1}{4}$ with some absolute constants  {$t_0, c>0$}: \begin{equation*}
\begin{split}
|x-x'|^{1+ct} \le |\Phi(t,x)-\Phi(t,x')|\le |x-x'|^{1-ct}
\end{split}
\end{equation*} and in particular, by taking $x' = 0$, \begin{equation}\label{eq:radial-bds}
\begin{split}
|x |^{1+ct} \le |\Phi(t,x) |\le |x |^{1-ct}.
\end{split}
\end{equation}


We now prove that for $0<t<t_0$ and $0<r<\dlt_0$, $\omg(t)\equiv 1$ on the region $$\bbA:=\{ (r,\tht): 0<r<\dlt_0, \frac{\pi}{20}<\tht<\frac{\pi}{4} \}$$ for $t_0$ and $\dlt_0$ taken sufficiently small. In the following we shall take them (in a way depending only on other absolute constants) smaller whenever it becomes necessary. Towards a contradiction, assume that there is a point $x_0\in[0,1]^2$ with $\omg_0(x_0)<1$ and $0<t<t_0$ that $\Phi(t,x_0)\in \bbA$. For $\omg_0(x_0)<1$, $x_0=(r_0,\tht_0)$ (in polar coordinates) should satisfy at least one of the following: (i) $r_0>\frac{1}{2}$, (ii) $\tht_0<r_0^\beta$, (iii) $\frac{\pi}{2} - \tht_0<r_0^\beta$. By arranging $\dlt_0,t_0$ small, we can exclude the possibility of (i) using \eqref{eq:radial-bds}. Now assume (ii) holds. We may further assume that $r_0<\sqrt{\dlt_0}$ by taking $t_0$ smaller if necessary. Writing $\Phi(t,x_0) = (r(t),\tht(t))$ in polar coordinates and applying \eqref{eq:log-Lip} with $x = \Phi(t,x_0)$ and $x' = (r(t),0)$, \begin{equation*}
\begin{split}
|u^\tht(t,\Phi(t,x_0))|\le Cr(t)\tht(t) \ln \left( \frac{1}{r(t)\tht(t)}\right). 
\end{split}
\end{equation*} Here $u^\tht = u \cdot e^\tht$ is the angular component of the velocity and we have used that it vanishes on the axes, due to the odd symmetry of $\omg$. Hence we can bound \begin{equation*}
\begin{split}
\left|\frac{d}{dt} \tht(t) \right| = \frac{|u^\tht(t,\Phi(t,x_0))|}{|\Phi(t,x_0)|} \le C\tht(t) \ln \left( \frac{1}{r(t)\tht(t)}\right)\le C\tht(t) \ln \left( \frac{1}{r_0^{1+ct}\tht(t)}\right). 
\end{split}
\end{equation*} Now we may take $t_0$ small that $1+ct<2$ for all $t\le t_0$ and then \begin{equation*}
\begin{split}
\frac{d}{dt} \ln\left( \frac{1}{\tht(t)} \right) \ge -C\left( \ln\left(\frac{1}{r_0}\right) + \ln\left(\frac{1}{\tht(t)}\right) \right),
\end{split}
\end{equation*} which gives \begin{equation*}
\begin{split}
 \ln\left( \frac{1}{\tht(t)} \right) & \ge e^{-Ct}  \ln\left( \frac{1}{\tht_0} \right) + (e^{-Ct}-1)  \ln\left( \frac{1}{r_0} \right)  \ge (1-Ct) \ln\left( \frac{1}{\tht_0} \right) -Ct \ln\left( \frac{1}{r_0} \right) \\
 &\ge  (1-Ct-\frac{Ct}{\beta}) \ln\left( \frac{1}{\tht_0} \right),
\end{split}
\end{equation*} where in the last step we have used that $\tht_0< r_0^\beta$. Therefore we can guarantee that for all $t\le t_0$ and $r_0<\dlt_0$, \begin{equation*}
\begin{split}
\tht(t) \le \tht_0^{1-Ct(1+\frac{1}{\beta})} < r_0^{\beta(1-Ct(1+\frac{1}{\beta}) )} < \dlt_0^{\frac{\beta}{2}(1-Ct(1+\frac{1}{\beta}) )} <\frac{\pi}{20}
\end{split}
\end{equation*} by taking $t_0,\dlt_0$ smaller if necessary. This gives the desired contradiction. The proof that (iii) is impossible can be done similarly. 

Let us follow the trajectory of the  {curve} $\{ (r_0,\frac{\pi}{2}-r_0^\beta)_* : 0<r_0< \dlt_0^4 \}$, on which $\omg_0\equiv 1$.  {Here and in the following, let us use the polar coordinates with notation $(r,\tht)_* = (r\cos\tht, r\sin\tht)$. For convenience, let us define}
$$(r(t),{\tht}(t))_* :=\Phi(t,(r_0,\frac{\pi}{2}-r_0^\beta)_*).$$ We can assume that on $[0,t_0]$, $r(t)<\dlt_0^2$ for any $r_0<\dlt_0^4$. Moreover, on the diagonal $\{ (a,a) : a < {\dlt_0^2}  \}$, we may compute that \begin{equation*}
\begin{split}
I(t,(a,a)) &= \frac{4}{\pi}\int_{ [2a,1]\times[2a,1]} \frac{y_1y_2}{|y|^4} \omg(t,y)dy \ge \frac{4}{\pi}\int_{\bbA \cap ([{2\dlt_0^2},1]\times[ {2\dlt_0^2},1] ) } \frac{y_1y_2}{|y|^4}dy  \\
&\ge c \int_{2\dlt_0^2}^{\dlt_0} \int_{2\dlt_0^2}^{y_1} \frac{y_2}{|y_1|^3} dy_2 dy_1 \ge c\ln\frac{1}{\dlt_0} - C,
\end{split}
\end{equation*}  {with some constants $c, C>0$,} using that $\bbA \cap ([{2\dlt_0^2},1]\times[ {2\dlt_0^2},1] ) $ contains the region $ \{ (y_1,y_2): 2\dlt_0^2<y_1<\dlt_0, 2\dlt_0^2<y_2<y_1 \}$. Hence, for $\dlt_0>0$ sufficiently small, the term $I(t,(a,a))$ can  {``dominate''} $B_j(t,(a,a))$ for $j = 1, 2$ in \eqref{eq:key}.  {To be precise, taking $\dlt_0>0$ smaller if necessary, it follows from the error estimate \eqref{eq:B-est} with $x = (a,a)$ that \begin{equation*}
	\begin{split}
		|B_j(t,(a,a))| \le \frac{1}{10}I(t,(a,a)), \qquad j=1, 2. 
	\end{split}
\end{equation*}} This implies that the velocity is pointing northwest on the diagonal, which guarantees in particular that $ {\tht}(t) > \frac{\pi}{4}$. Now, along the trajectory $(r(t),\tht(t))_*$, we compute using \eqref{eq:key} that \begin{equation*}
\begin{split}
\dot{\tht}(t) = \frac{u^\tht(t,(r(t),\tht(t))_*)}{|r(t)|} & = \frac{1}{2}\sin(2\tht(t))\left( 2I + B_1 + B_2 \right) \\
& \ge \sin(2\tht(t))\left( I(t,(r(t),\tht(t))_*)  - C(1+\ln \frac{1}{\frac{\pi}{2}-\tht(t)}) \right).
\end{split}
\end{equation*} Introducing $\bar{\tht}(t) = \frac{\pi}{2}-\tht(t)$, the above inequality becomes \begin{equation}\label{eq:lb}
\begin{split}
\dot{\bar{\tht}}(t) \le - 2\bar{\tht}(t) \left(  {I(t,(r(t),\tht(t))_*)  } - C(1 + \ln\frac{1}{\bar{\tht}(t)}) \right). 
\end{split}
\end{equation} We now estimate  {$I(t,(r(t),\tht(t))_*)$} from below, using $\omg(t) \equiv 1$ on $\bbA$. Note that \begin{equation*}
\begin{split}
\bbA \cap ([{2r(t)\cos(\tht(t))},1]\times[ {2r(t)\sin(\tht(t))},1] )\supset \{ (y_1,y_2): 3r(t)<y_1<\dlt_0, 3r(t)<y_2<y_1 \}
\end{split}
\end{equation*} and using $r(t)\le r_0^{1-ct}$, \begin{equation*}
\begin{split}
I(t,(r(t),\tht(t))_*) \ge c(1-ct)\ln\left( \frac{1}{r_0} \right)= \frac{c}{\beta}(1-ct)\ln\left( \frac{1}{\bar{\tht}_0} \right)
\end{split}
\end{equation*} where the constant $c>0$ depends on $\dlt_0$. Hence, as long as we have \begin{equation}\label{eq:ansatz}
\begin{split}
\frac{c}{\beta}(1-ct)\ln\left( \frac{1}{\bar{\tht}_0} \right) > 2C(1 + \ln\frac{1}{\bar{\tht}(t)})
\end{split}
\end{equation}  where $C>0$ is the constant from \eqref{eq:lb}, we have \begin{equation*}
\begin{split}
\dot{\bar{\tht}}(t) \le -\frac{c}{\beta}  {(1-ct)}\bar{\tht}(t) \ln \frac{1}{\bar{\tht}_0}
\end{split}
\end{equation*} for $t\le t_0$ by taking $t_0$ smaller and therefore \begin{equation*}
\begin{split}
\bar{\tht}(t) \le \bar{\tht}_0^{1+\frac{ct}{\beta}} = r_0^{\beta+ct} \le (r(t))^{(1-ct)(\beta+ct)}.
\end{split}
\end{equation*} Note that for $\beta>0$ sufficiently small, \eqref{eq:ansatz} is satisfied for $t\in[0,t_0]$, again by taking $t_0$ smaller  {to satisfy $ct\le ct_0<\frac18$,} if necessary. Moreover,  {$(1-ct)(\beta+ct) \ge \beta + c_0t$ with $c_0 = \frac{c}{8}$ if $0<\bt<\frac12$ and $ct_0<\frac18$}. Recall that \begin{equation*}
\begin{split}
p_0 = 1 + \frac{1}{1+\beta} > p^* 
\end{split}
\end{equation*} and by taking $2>p^*$ closer to 2, we can guarantee \eqref{eq:ansatz}. Recalling that $\omg(t,\cdot)\equiv 1$ on $(r(t),\frac{\pi}{2}-\bar{\tht}(t))$ and applying Lemma \ref{lem:lvl-Sob} gives that \begin{equation*}
\begin{split}
\omg(t)\notin W^{1,q(t)}(\bbT^2),\quad q(t) = 1 + \frac{1}{1+\beta+c_0t} .
\end{split}
\end{equation*} This finishes the proof.

\subsection{Discussions}\label{subsec:disc}

Let us close the paper with commenting on related issues that we have not touched upon so far.

\subsubsection*{Loss of regularity with $H^1$ initial vorticity.} It is not clear to us whether it is possible for $\omg_0\in H^1\cap L^\infty$ to lose $W^{1,p}$ regularity with time. Let us briefly discuss the main difficulty. Note that as a consequence of Lemma \ref{lem:lvl-Sob}, $\omg_0 \notin H^1(\bbT^2)$ if $f$ takes on different constant values along two half-lines emanating from the origin. That is, insisting on the odd-odd scenario forces the vorticity to vanish near the origin. A natural choice, which was used in \cite{EJ} is to take roughly $\omg_0(r,\tht)\sim (\ln\frac{1}{r})^{-\gmm}\sin(2\tht)$. Then one can check that the corresponding velocity gradient satisfies $|\nb u_0(r,\tht)|\sim (\ln\frac{1}{r})^{1-\gmm}$, and even the passive transport with $u_0$ is not strong enough to remove the vorticity from $W^{1,p}$ with any $p<2$. 
After the completion of this work, we have learned about an interesting recent paper \cite{Hung} which proves that (see Theorem 3.1 therein) for $\omega_0\in H^1\cap C^0$, we have $\omega(t) \in W^{1,p} \cap W^{\alp,2}$ for any $p<2$ and $\alp<1$. That is, continuity (opposed to mere boundedness) of the vorticity makes the situation completely different. 

\subsubsection*{Loss of regularity in negative H\"older spaces.} 

Another natural question to ask is whether there exists initial vorticity $\omg_0\in C^{-\alp}(\bbT^2)$ with a solution $\omg(t) \notin C^{-\alp}(\bbT^2)$ for $t>0$. There are serious essential difficulties in proving such a statement; now the initial vorticity is necessarily unbounded and the uniqueness of solution is not guaranteed (see very recent progress in \cite{V1,V2,Bre,BV,Elling2016}). On the other hand, any $L^{p,q}$ norms of the vorticity is preserved for any solution, which makes it hard to lose H\"older regularity.

{\subsubsection*{Loss of regularity for active scalars.} 

It will be an interesting problem to extend loss of regularity to active scalar equations, most notably to the case of SQG (surface quasi-geostrophic) equations. In this case, an analogue of Yudovich theorem is not available and therefore one should develop a new strategy.  
}

\subsection*{Acknowledgement}

\noindent  IJ has been supported  by the New Faculty Startup Fund from Seoul National University, the Science Fellowship of POSCO TJ Park Foundation, and the National Research Foundation of Korea grant (No. 2019R1F1A1058486). {We sincerely thank the anonymous referees for their kind words and numerous suggestions which significantly improved the readability of the paper.} The author states that there is no conflict of interest.

\bibliographystyle{amsplain}

\begin{thebibliography}{10}
\bibitem{BT}
Claude Bardos and Edriss~S. Titi, \emph{Loss of smoothness and energy
	conserving rough weak solutions for the {$3d$} {E}uler equations}, Discrete
Contin. Dyn. Syst. Ser. S \textbf{3} (2010), no.~2, 185--197. \MR{2610558}

\bibitem{BT2}
K.~Bardos and \`E.~S. Titi, \emph{Euler equations for an ideal incompressible
	fluid}, Uspekhi Mat. Nauk \textbf{62} (2007), no.~3(375), 5--46. \MR{2355417}

\bibitem{BL1}
Jean Bourgain and Dong Li, \emph{Strong ill-posedness of the incompressible
	{E}uler equation in borderline {S}obolev spaces}, Invent. Math. \textbf{201}
(2015), no.~1, 97--157. \MR{3359050}

\bibitem{BL2}
\bysame, \emph{Strong illposedness of the incompressible {E}uler equation in
	integer {$C^m$} spaces}, Geom. Funct. Anal. \textbf{25} (2015), no.~1, 1--86.
\MR{3320889}

\bibitem{BL3D}
\bysame, \emph{Strong ill-posedness of the 3{D} incompressible {E}uler equation
	in borderline spaces}, International Mathematics Research Notices (2019).

\bibitem{Bre}
Alberto Bressan and Wen Shen, \emph{A posteriori error estimates for
	self-similar solutions to the {E}uler equations}, Discrete Contin. Dyn. Syst.
\textbf{41} (2021), no.~1, 113--130. \MR{4182316}

\bibitem{Hung}
Nguyen~Q. Brue, E., \emph{Sobolev estimates for solutions of the transport
	equation and ode flows associated to non-lipschitz drifts}, Math. Ann.
(2020).

\bibitem{BV}
Tristan Buckmaster and Vlad Vicol, \emph{Convex integration and phenomenologies
	in turbulence}, EMS Surv. Math. Sci. \textbf{6} (2019), no.~1, 173--263.
\MR{4073888}

\bibitem{DiPM}
Ronald~J. DiPerna and Andrew~J. Majda, \emph{Oscillations and concentrations in
	weak solutions of the incompressible fluid equations}, Comm. Math. Phys.
\textbf{108} (1987), no.~4, 667--689. \MR{877643}

\bibitem{EJ2}
Tarek~M. Elgindi and In-Jee Jeong, \emph{On singular vortex patches, {I}:
	Well-posedness issues}, Memoirs of the AMS, to appear, arXiv:1903.00833.

\bibitem{EJ}
\bysame, \emph{Ill-posedness for the {I}ncompressible {E}uler {E}quations in
	{C}ritical {S}obolev {S}paces}, Ann. PDE \textbf{3} (2017), no.~1, 3:7.
\MR{3625192}

\bibitem{EJSVP2}
\bysame, \emph{On singular vortex patches, {II}: long-time dynamics}, Trans.
Amer. Math. Soc. \textbf{373} (2020), no.~9, 6757--6775. \MR{4155190}

\bibitem{EM1}
Tarek~M. Elgindi and Nader Masmoudi, \emph{{$L^\infty$} ill-posedness for a
	class of equations arising in hydrodynamics}, Arch. Ration. Mech. Anal.
\textbf{235} (2020), no.~3, 1979--2025. \MR{4065655}

\bibitem{Elling2016}
Volker Elling, \emph{Self-similar 2d {E}uler solutions with mixed-sign
	vorticity}, Comm. Math. Phys. \textbf{348} (2016), no.~1, 27--68.
\MR{3551260}

\bibitem{Jthesis}
In-Jee Jeong, \emph{Dynamics of the incompressible {E}uler equations at
	critical regularity}, Ph.D. thesis, Princeton University, 2017.

\bibitem{JKi}
In-Jee Jeong and Junha Kim, \emph{Strong ill-posedness of {S}{Q}{G} in critical
	{S}obolev spaces}, arXiv:2107.07739.

\bibitem{JY2}
In-Jee Jeong and Tsuyoshi Yoneda, \emph{Enstrophy dissipation and vortex
	thinning for the incompressible 2{D} {N}avier-{S}tokes equations},
Nonlinearity \textbf{34} (2021), no.~4, 1837--1853. \MR{4246445}

\bibitem{JY}
\bysame, \emph{Vortex stretching and enhanced dissipation for the
	incompressible 3{D} {N}avier-{S}tokes equations}, Math. Ann. \textbf{380}
(2021), no.~3-4, 2041--2072. \MR{4297205}

\bibitem{KS}
Alexander Kiselev and Vladimir {\v{S}}ver{\'a}k, \emph{Small scale creation for
	solutions of the incompressible two-dimensional {E}uler equation}, Ann. of
Math. (2) \textbf{180} (2014), no.~3, 1205--1220. \MR{3245016}

\bibitem{MB}
Andrew~J. Majda and Andrea~L. Bertozzi, \emph{Vorticity and incompressible
	flow}, Cambridge Texts in Applied Mathematics, vol.~27, Cambridge University
Press, Cambridge, 2002. \MR{1867882}

\bibitem{MP}
Carlo Marchioro and Mario Pulvirenti, \emph{Mathematical theory of
	incompressible nonviscous fluids}, Applied Mathematical Sciences, vol.~96,
Springer-Verlag, New York, 1994. \MR{1245492}

\bibitem{MY}
Gerard Misio\l~ek and Tsuyoshi Yoneda, \emph{Continuity of the solution map of
	the {E}uler equations in {H}\"{o}lder spaces and weak norm inflation in
	{B}esov spaces}, Trans. Amer. Math. Soc. \textbf{370} (2018), no.~7,
4709--4730. \MR{3812093}

\bibitem{V1}
Misha Vishik, \emph{Instability and non-uniqueness in the {C}auchy problem for
	the {E}uler equations of an ideal incompressible fluid. {P}art {I}},
arXiv:1805.09426.

\bibitem{V2}
\bysame, \emph{Instability and non-uniqueness in the {C}auchy problem for the
	{E}uler equations of an ideal incompressible fluid. {P}art {I}{I}},
arXiv:1805.09440.

\bibitem{Y1}
V.~I. Yudovich, \emph{Non-stationary flows of an ideal incompressible fluid},
Z. Vycisl. Mat. i Mat. Fiz. \textbf{3} (1963), 1032--1066. \MR{0158189}

\bibitem{Z}
Andrej Zlato{\v{s}}, \emph{Exponential growth of the vorticity gradient for the
	{E}uler equation on the torus}, Adv. Math. \textbf{268} (2015), 396--403.
\MR{3276599}

\bibitem{Zm}
Andrej Zlato\v{s}, \emph{On the rate of merging of vorticity level sets for the
	2{D} {E}uler equations}, J. Nonlinear Sci. \textbf{28} (2018), no.~6,
2329--2341. \MR{3867645}

\end{thebibliography}


\end{document}